\documentclass[11pt]{amsart}
\usepackage{sfilip_presets,amsmath,amsthm}
\usepackage[margin=1.51in]{geometry}
\usepackage[title=normal, sections=normal, subtle, charwidths=normal, tracking=normal]{savetrees}
\usepackage{graphicx,mathabx,mathtools}
\usepackage{tikz-cd}
\usepackage{lmodern}
\usepackage[T1]{fontenc}
\usepackage[stretch=10]{microtype}
\newcommand{\vertiii}[1]
{{\left\vert\kern-0.25ex\left\vert\kern-0.25ex\left\vert #1 \right\vert\kern-0.25ex\right\vert\kern-0.25ex\right\vert}}

\makeatletter
\renewcommand{\paragraph}{%
\@startsection {paragraph}{4}
{\z@} \z@ {-\fontdimen 2\font }\bfseries
}
\makeatother

\newif\ifdebug
\debugfalse

\ifdebug
  \usepackage[status=draft,author=]{fixme}
  \fxsetup{inline, marginclue,theme=color}
\else
  \usepackage[status=final, author=]{fixme}
\fi

\usepackage{aliascnt}  

\usepackage{amsthm}
\numberwithin{equation}{section}

\def\subsek~{\S{}}

\newtheorem{theorem}{Theorem}[section]

\newaliascnt{lemma}{theorem}  
  
\aliascntresetthe{lemma}

\newaliascnt{proposition}{theorem}
\newtheorem{proposition}[proposition]{Proposition}
\aliascntresetthe{proposition}

\newaliascnt{corollary}{theorem}  
\newtheorem{corollary}[corollary]{Corollary}  
\aliascntresetthe{corollary}

\newaliascnt{example}{theorem}  
\newtheorem{example}[example]{Example}  
\aliascntresetthe{example}

\newaliascnt{exercise}{theorem}  
  
\aliascntresetthe{exercise}

\theoremstyle{remark}
\newaliascnt{remark}{theorem}  
\newtheorem{remark}[remark]{Remark}  
\aliascntresetthe{remark}

\theoremstyle{definition}
\newaliascnt{definition}{theorem}  
\newtheorem{definition}[definition]{Definition}  
\aliascntresetthe{definition}

\def\equationautorefname#1#2\null{%
  Eq.#1(#2\null)%
}
\usepackage{etoolbox}

\usepackage{needspace}
\AtBeginEnvironment{theorem}{\Needspace{5\baselineskip}}

\PassOptionsToPackage{hyphens}{url}\usepackage{hyperref}
\makeatletter
\g@addto@macro{\UrlBreaks}{\UrlOrds}
\makeatother

\usepackage{color}
\definecolor{darkred}{rgb}{0.4,0,0}
\definecolor{darkgreen}{rgb}{0,0.5,0}
\definecolor{darkblue}{rgb}{0,0,0.4}
\hypersetup{
    pdftitle={Splitting Mixed Hodge Structures over Affine Invariant Manifolds},	
    pdfauthor={Simion Filip},	
    pdfsubject={math},		
    pdfkeywords={},	
    pdfnewwindow=true,		
    colorlinks=true,		
    linkcolor=darkblue,		
    citecolor=darkred,		
    filecolor=darkblue,		
    urlcolor=darkblue,		
    pdfborder={0 0 0},
    breaklinks=true}

\title[Splitting MHS over Affine Invariant Manifolds]{Splitting Mixed Hodge Structures over Affine Invariant Manifolds}

\thanks{{Revised \today}}
\date{November 2013}

\author{{ Simion Filip}
}

\begin{document}
\begin{abstract}
We prove that affine invariant manifolds in strata of flat surfaces are algebraic varieties.
The result is deduced from a generalization of a theorem of M\"{o}ller.
Namely, we prove that the image of a certain twisted Abel-Jacobi map lands in the torsion of a factor of the Jacobians.
This statement can be viewed as a splitting of certain mixed Hodge structures.
\end{abstract}

\maketitle


\ifdebug
  \listoffixmes
\fi

\section{Introduction}

Let $(X,\omega)$ be a Riemann surface with a holomorphic $1$-form on it.
The set of all such pairs forms an algebraic variety $\cH(\kappa)$ called a stratum, where $\kappa$ encodes the multiplicities of the zeros of $\omega$.
The stratum carries a natural action of the group $\SL_2\bR$ which is of transcendental nature.


A stratum $\cH(\kappa )$ has natural charts to complex affine spaces.
The coordinates are the periods of $\omega$ on $X$, thus in $\bC$.
After identifying $\bC$ with $\bR^2$, the action of $\SL_2\bR$ is the standard one on each coordinate individually.
In particular, the statum carries a natural Lebesgue-class measure which is invariant under $\SL_2\bR$.
The finiteness of the measure was proved by Masur \cite{Masur} and Veech \cite{Veech}.

The action of $\SL_2\bR$ and knowledge of invariant measures can be applied to study other dynamical systems.
For interval exchange transformations this started in the work of Masur and Veech \cite{Masur, Veech}.
A starting point for applications to polygonal billiards was by Kerckhoff-Masur-Smillie in \cite{KMS}.
Some recent applications involve a detailed analysis of the wind-tree model by Hubert-Leli{\`e}vre-Troubetzkoy \cite{windtree}.
For a comprehensive introduction to the subject, see the survey of Zorich \cite{Zorich_survey}.

For more precise applications, especially to concrete examples, one needs to understand all possible invariant measures.
For instance, polygonal billiards correspond to a set of Masur-Veech measure zero.

Recent results of Eskin and Mirzakhani \cite{EM} show that finite ergodic invariant measures are rigid and in particular are of Lebesgue class and supported on smooth manifolds.
In further work with Mohammadi \cite{EMM} they show that many other analogies with the homogeneous setting and Ratner's theorems hold.

The $\SL_2\bR$-invariant measures give rise to \emph{affine invariant manifolds}.
These are complex manifolds which are given in local period coordinates by linear equations.
It was shown by Wright \cite{Wright_field} that the linear equations can be taken with coefficients in a number field.

Note that \emph{finite} $\SL_2\bR$-invariant measures are supported on real codimension $1$ hypersurfaces inside affine manifolds, the issue arising from the action of scaling by $\bR^\times$.
The affine invariant manifolds are then closed $\GL_2\bR$-invariant sets.
In fact, by \cite{EMM} the closure of any $\GL_2\bR$-orbit is an affine manifold.

Period coordinates are transcendental and so apriori affine manifolds, which are given by linear equations, are only complex-analytic submanifolds.
In this paper, we prove the following result (see \autoref{thm:algebraic}).

\begin{theorem}
 Affine invariant manifolds are algebraic subvarieties of the stratum $\cH(\kappa)$, defined over $\conj{\bQ}$.
\end{theorem}

I am grateful to Curtis McMullen for suggesting the next result (see also \cite[Theorem 1.1]{McMullen_classification} for a much more precise result in genus $2$).

\begin{corollary}[see \autoref{rmk:Teich_disk}]
\label{cor:Teich_disk}
Let $\bH\to \cM_g$ be a \Teichmuller disk.
Then its closure (in the standard topology) inside $\cM_g$ is an algebraic subvariety of $\cM_g$.
\end{corollary}

The lowest-dimensional affine manifolds are \Teichmuller curves.
It was proved they are algebraic by Smillie and Weiss \cite[Proposition 8]{Smillie_Weiss};
a different sketch of proof (attributed to Smillie) is in \cite{Veech_closure}.
That they are defined over $\conj{\bQ}$ was proved by McMullen \cite{McMullen_def_Q}.
It was proved by M{\"o}ller in \cite{Moller_VHS} that they are defined over $\conj{\bQ}$ after embedding into a moduli space of abelian varieties.

\Teichmuller curves and higher-dimensional $\SL_2\bR$-invariant loci also have interesting arithmetic properties.
McMullen has related in \cite{McMullen_Hilbert_mod_surF} \Teichmuller curves in genus $2$ with real multiplication.
He also gave further constructions using Prym loci \cite{McMullen_prym}.
In genus $2$ algebraicity follows from a complete classification of invariant loci by McMullen \cite{McMullen_classification}.
In the stratum $\cH(4)$ algebraicity is known by results of Aulicino, Nguyen and Wright \cite{ANW_H^4, NW}.
Lanneau and Nguyen have also done extensive work on Prym loci in genus $3$ and $4$ \cite{LN_periodicity, LN_Prym_g3_G4}.

Techniques from variations of Hodge structures were introduced by M{\"o}ller, starting in \cite{Moller_VHS}.
In particular, he showed that \Teichmuller curves always parametrize surfaces with Jacobians admitting real multiplication on a factor.
He also showed that over a \Teichmuller curve, the Mordell-Weil group of the corresponding factor is finite \cite{Moller_torsion}.
In particular, zeros of the $1$-form are torsion under the Abel-Jacobi map.
See also \cite{Moller_Linear_mnfds} for further results.

The results in \cite{ssimple} show that on affine manifolds, the topological decomposition of cohomology (e.g. the local systems from \cite{Wright_field}) are compatible with the Hodge structures.
As a consequence, affine manifolds also parametrize Riemann surfaces with non-trivial endormophisms, typically real multiplication on a factor (see \cite[Thm. 7.3]{ssimple}).

This paper extends M{\"o}ller's torsion result to affine manifolds.
The precise statement and definitions are in \autoref{subsec:combinig_splittings} and \autoref{thm:torsion}.

\begin{theorem}
 Let $\cM$ be an affine invariant manifold, parametrizing Riemann surfaces with real multiplication by the order $\cO$ on a factor of the Jacobians.
 Then $\cM$ carries a natural local system $\Lambda$ of $\cO$-linear combinations of the zeros of the $1$-form (see \autoref{eqn:Lambda_iota_def}) and a twisted Abel-Jacobi map (see \autoref{def:twisted_cycle})
 $$
 \nu : \Lambda \to \Jac_\bZ\left(\oplus_\iota H^1_\iota\right)
 $$
 The range is the factor of Jacobians admitting real multiplication (see \autoref{eqn:H1_decomp}).
 
 Then the image of $\nu$ always lies in the torsion of the abelian varieties.
\end{theorem}

\begin{remark}
\leavevmode
\begin{itemize}
 \item[(i)] The expression ``real multiplication by $\cO$'' is used in a rather loose sense.
 It means that the ring $\cO$ maps to the endomorphisms of a factor of the Jacobian.
 The factor of the Jacobian is always nontrivial, as it contains at least the part coming from the $1$-form $\omega$.
 The ring $\cO$ could be $\bZ$, and the factor could be the entire Jacobian.
 
 \item[(ii)] The local system $\Lambda$ can be trivialized on a finite cover of the stratum, and is defined as follows.
 The tangent space to the stratum contains the relative cohomology classes that vanish on absolute homology, denoted $W_0$.
 The tangent space $T\cM$ of the affine manifold intersects it in a sublocal system (over $\cM$), denoted $W_0\cM$.
 The dual of $W_0$, denoted $\check{W_0}$, is canonically identified with linear combinations of the zeros of the $1$-form with zero total weight.
 Then $\Lambda$ is an $\cO$-submodule of $(W_0\cM)^\perp\subset \check{W_0}$, i.e. of the annihilator of $W_0\cM$.
 
 \noindent Indeed, by results of Wright \cite{Wright_field}, $W_0\cM$ and thus $(W_0\cM)^\perp$ are defined over $k$ - the field giving real multiplication.
 Since $\check{W}_0$ carries a $\bZ$-structure, extending scalars to $\cO$, define $\Lambda:=\check{W}_0(\cO)\cap (W_0\cM)^\perp(k)$, where $A(R)$ denotes the $R$-points of $A$.
 \item[(iii)]
 When $W_0\cM$ is empty, e.g. for \Teichmuller curves, $\Lambda$ coincides with $\check{W}_0(\cO)$.
 In particular, it contains (up to finite index) all the $\bZ$-linear combinations of zeroes of the $1$-form, with total weight zero; on them $\nu$ is the usual Abel-Jacobi map.
 The extension of $\nu$ to $\cO$-linear combinations of zeros uses the $\cO$-action on the Jacobian factor.
\end{itemize}
\end{remark}

\begin{remark}
 The result on torsion can be described concretely using periods of $1$-forms.
 Note that it refers, in particular, to $1$-forms other than $\omega$; describing them using the flat structure does not seem immediate.
 
 Let $\{r_j\}_j$ be formal integral combinations of the zeroes of $\omega$, such that the coefficients of each $r_j$ add up to $0$.
 They can be lifted to actual relative cycles on the surface, denoted $r_j'$ (these are now actual curves that connect zeroes of $\omega$).
 Let also $\{a_i,b_i\}_{i=1}^g$ denote an integral basis of the first homology of the surface.
 
 Suppose $\sum_j c^j r_j$ is an element of $\Lambda$, where $c_j$ are elements of $\cO$. 
 The condition that its image is torsion under the twisted cycle map $\nu$ is equivalent to the following:
 There exist $\alpha_i,\beta_i\in \bQ$ such that whenever $\omega_l\in H^{1,0}_{\iota_l}$ is a holomorphic $1$-form, we have
 \begin{align}
 \sum_j \iota_l (c^j)\int_{r_j'}\omega_l = \sum_i \left(\alpha_i\int_{a_i} \omega_l + \beta_i \int_{b_i}\omega_l\right)
 \end{align}
 In other words, the absolute and relative periods of $\omega_l$ satisfy some linear relations.
 The coefficients $\iota_l(c^j)$ vary with the embedding $\iota_l$ corresponding to the subspace in which $\omega_l$ lives.
\end{remark}

\begin{remark}
\label{rmk:Teich_disk}
 The algebraicity result also applies to strata of quadratic differentials.
 Indeed, these embed via the double-covering construction to strata of holomorphic $1$-forms.
 An affine invariant submanifold of a stratum of quadratic differentials can thus be viewed as one in a stratum of $1$-forms.
 In particular, for \autoref{cor:Teich_disk}, it suffices to restrict to \Teichmuller disks coming from holomorphic $1$-forms.
 
 Over the moduli space of curves $\cM_g$ we have the Hodge bundle $\cH_g$ whose fibers are the holomorphic $1$-forms.
 We can consider the projectivization of $\cH_g$ and therefore the proper projection $\bP\cH_g\to \cM_g$.
 The stratum $\cH(\kappa)$ is a subvariety of $\cH_g$, and we can also quotient by the $\bC^\times$-action to obtain $\bP\cH(\kappa)\subseteq \bP\cH_g$.
 
 Now a \Teichmuller disk $f:\bH\to \cM_g$ as in \autoref{cor:Teich_disk} lifts to $\tilde{f}:\bH \to \bP\cH(\kappa)\subseteq \bP\cH_g$.
 Combining Theorem 2.1 from \cite{EMM} and \autoref{thm:algebraic} we find that the (Zariski and usual) closure of $\tilde{f}(\bH)$ in $\cH(\kappa)$ is an algebraic variety $\bP\cM$.
 We can further take its (Zariski and usual) closure inside $\bP\cH_g$ to find that it is also an algebraic variety $\bP\overline{\cM}\subseteq \bP\cH_g$.
 Note that the Zariski and usual closure of a quasi-projective set coincide.
 
 Now, the topological closure of $f(\bH)$ will agree with the projection of the topological closure of $\tilde{f}(\bH)$ inside $\bP\cH_g$, which is $\bP\overline{\cM}$.
 This follows from the properness of $\bP\cH_g\to \cM_g$.
 Properness also ensures that the projection of an algebraic variety is still a variety, so \autoref{cor:Teich_disk} follows.
\end{remark}

\paragraph{Outline of the paper.}
\autoref{sec:alg_special_case} proves in a special case that affine invariant manifolds are algebraic.
This special case occurs when the tangent space of the affine manifold contains all relative cohomology classes.
The proof only uses results from \cite{ssimple}.

\autoref{sec:MHS} contains basic definitions and constructions about mixed Hodge structures.
We only describe the small part of the theory that is necessary for our arguments.
The proofs in later sections use this formalism, and we also include some concrete examples throughout.
One does not need to be acquainted with the full theory to follow the arguments.

\autoref{sec:splittings} contains the main technical part.
It proves that certain sequences of mixed Hodge structures are \emph{split}, i.e. as simple as possible.
This uses the negative curvature properties of Hodge bundles.

\autoref{sec:algebraicity_torsion} combines the previous results to deduce the Torsion \autoref{thm:torsion}.
This is then used to prove the Algebraicity \autoref{thm:algebraic}.

\begin{remark}[\textbf{Self-intersections}]
\label{rmk:self_inters}
Affine manifolds are only immersed in a stratum (see \cite[Def.~1.1]{EM}), and could have locally finitely many self-intersecting sheets.
Thus, any affine manifold $\cM$ can be written as the union $\cM=\cM_0\coprod \cM'$ where $\cM_0$ is a smooth open subset of $\cM$ and $\cM'$ is a lower-dimensional proper closed $\GL_2^+(\bR)$-invariant affine manifold (possibly disconnected or with self-intersections).
Moreover the topological closure of $\cM_0$ contains $\cM'$.
Since $\dim_\bC\cM' < \dim_\bC \cM$, it follows that $\cM_0$ is connected if $\cM$ is.

The arguments in \autoref{sec:alg_special_case} and \autoref{sec:algebraicity_torsion} about algebraicity apply locally on $\cM_0$ and identify it (locally) with a quasi-projective variety.
Thus, assuming by induction that $\cM'$ is quasi-projective, they show that $\cM_0$ is quasi-projective inside $\cH\setminus \cM'$.
Again, since $\cM'$ is quasi-projective and contained in the (topological) closure of $\cM_0$, it follows that $\cM$ is quasi-projective inside $\cH$. 
\end{remark}

\paragraph{Orbifolds.}
All the arguments are made in some finite cover of a stratum, to avoid orbifold issues.
In particular, period coordinates are well-defined and the zeros of the $1$-form are labeled.
The results are invariant under passing to finite covers.

\paragraph{Acknowledgments.}
I would like to thank my advisor Alex Eskin, who was very helpful and supportive at various stages of this work, and in particular about the paper \cite{ssimple} whose methods are used here.
I have also benefited a lot from conversations with Madhav Nori, especially on the topic of mixed Hodge structures.

I also had several conversations on the topic of algebraicity with Alex Eskin and Alex Wright.
In particular, Alex Wright explained his (unpublished) result that the torsion and real multiplication theorems of M\"{o}ller characterize \Teichmuller curves and suggested that finding and proving some generalization of the torsion theorem could imply algebraicity.
I have also discussed and received very useful feedback on an earlier draft of this paper from both Alex Eskin and Alex Wright.
I am very grateful for their feedback and the numerous insights they shared with me.

I am also grateful to Curtis McMullen for suggesting \autoref{cor:Teich_disk}.
\section{Algebraicity in a particular situation}
\label{sec:alg_special_case}

In this section, we prove a special case of algebraicity.
It only requires results of \cite{ssimple}.
In this special case the location of the tangent space to an affine manifold can be described precisely.

\paragraph{Setup.} 
Consider an affine invariant manifold $\cM$ in some stratum $\cH(\kappa)$.
We omit $\kappa$ from the notation, and refer to the stratum as $\cH$.
Let $T\cH$ be the tangent bundle of the stratum, and let $W_0\subset T\cH$ be the subbundle corresponding to the purely relative cohomology classes.
The survey of Zorich \cite[Section 3]{Zorich_survey} provides a clear and detailed exposition of these objects.

The purpose of this section is to prove the following result.

\begin{proposition}
\label{prop:alg_special_case}
 Suppose that everywhere on $\cM$ we have $W_0\subset T\cM$, where $T\cM$ denotes the tangent bundle of $\cM$.
 
 Then $\cM$ is a quasi-projective algebraic subvariety of $\cH$.
\end{proposition}

Before proceeding to the proof, we recall some facts about the local structure on a stratum $\cH$.
In particular, we discuss the way in which the relative cohomology groups $H^1(X,(\omega)_{red};\bC)$ provide local coordinates.
These arise from the Gauss-Manin connection and the tautological section, which assigns to $(X,\omega)$ the cohomology class of $\omega$.

\paragraph{Some preliminaries.}
 As explained in \autoref{rmk:self_inters}, it suffices to argue locally in the open part of $\cM$ where there are no self-intersections.
 
 Focus on a small neighborhood in $\cH$ of some $(X_0,\omega_0)\in\cM$, denoted $N_\epsilon(X_0,\omega_0)$.
 Introduce the identification coming from parallel transport (i.e. the Gauss-Manin connection) on relative cohomology
 $$
 GM_{(X,\omega)}:H^1(X,(\omega)_{red};\bC)\toisom H^1(X_0,(\omega_0)_{red};\bC)
 $$
 This is defined for all $(X,\omega)\in N_\epsilon(X_0,\omega_0)$ and $(\omega)_{red}$ denotes the zeroes of $\omega$, forgetting the multiplicities (i.e. the reduced divisor).
 
 Recall that for $(X,\omega)$ we have the natural element $\omega\in H^1(X,(\omega)_{red};\bC)$ viewing $\omega$ as a relative cohomology class.
 Call this the \emph{tautological section} $\omega:\cH\to T\cH$.
 
 Period coordinates are then described by the tautological section:
 \begin{align}  
  \Pi:N_\epsilon(X_0,\omega_0)&\to H^1(X_0,(\omega_0)_{red};\bC) \label{eqn:per_coord}\\  
  (X,\omega) &\mapsto GM_{(X,\omega)}\omega\notag
 \end{align}
 Recall that we have a short exact sequence of vector bundles over $\cH$
 $$
 0\to W_0\to T\cH \overset{p}{\to}H^1\to 0
 $$
 Here $H^1$ is the bundle of absolute first cohomology of the underlying family of Riemann surfaces, and $W_0$ is the purely relative part as before.
 
 It is proved by Wright in \cite[Theorem 1.5]{Wright_field} that over $\cM$ the local system of cohomology decomposes as
 \begin{align}
 \label{eqn:H1_decomp}
 H^1=\left(\bigoplus_{\iota} H^1_{\iota}\right)\oplus V
 \end{align}
 The summation is over embeddings $\iota$ of a number field $k$ in $\bC$, and there is a distinguished real embedding $\iota_0$.
 We then have $p(T\cM)=H^1_{\iota_0}$.

\begin{proof}[Proof of \autoref{prop:alg_special_case}]
 In period coordinates in the neighborhood $N_\epsilon(X_0,\omega_0)$ the property $p(T\cM)=H^1_{\iota_0}$ translates to the statement
 \begin{align}
 \label{eqn:loc_def_M}
 \Pi(\cM\cap N_\epsilon) \subseteq p^{-1}(H^1_{\iota_0})_{(X_0,\omega_0)}\subseteq T\cH_{(X_0,\omega_0)}
 \end{align}
 Because $T\cM$ contains all the purely relative cohomology classes, locally $\cM$ coincides with an open subset of the middle space above.
 
 An equivalent way to state the above local description of $\cM$ is to say
 \begin{align}
 \label{eqn:cM_subset_eigenform_locus}
 \cM\cap N_\epsilon =\lbrace{(X,\omega)\; \vert \;p(GM_{(X,\omega)}\omega)\in H^1_{\iota_0} (X_0,\omega_0)\rbrace}
 \end{align}
 So over $\cM$ the section $p(\omega)$ must belong to the local system $H^1_{\iota_0}$.
 
 Note the local system $H^1_{\iota_0}$ cannot be globally defined on $\cH$.
 However, in the neighborhood $N_\epsilon$ one can still define it using the Gauss-Manin connection.
 
\noindent\textbf{Checking algebraicity.}
 We also know by \cite[Thm. 1.6, Thm. 7.3]{ssimple} that each $H^1_\iota$ carries a Hodge structure.
 Moreover, for each $a\in k$ we have the operator
 \begin{align}
 \rho(a)=\left(\bigoplus_\iota \iota(a)\right)\oplus 0
 \end{align}
 which acts by the corresponding scalar on each factor of the decomposition \eqref{eqn:H1_decomp}.
 These operators give real multiplication on the Jacobians, with $\omega$ as an eigenform.

 As a consequence, there is an order $\cO\subset k$ which acts by genuine endomorphisms of the Jacobian factor obtained via \eqref{eqn:H1_decomp}.
 Moreover, the isomorphism class of the corresponding $\bZ$-lattice, viewed as an $\cO$-module, is constant on $\cM$.
 Recall that we are working on the open subset of $\cM$ where there are no self-intersections, and this is still connected if $\cM$ is (see \autoref{rmk:self_inters}).
 
 Define $\cN'$ to be the locus in $\cH$ of $(X,\omega)$ which admit real multiplication by $\cO$ on a factor of the Jacobian, with the same $\cO$-module structure on the integral lattice of the factor, and with $\omega$ as an eigenform.
 This is a finite union of algebraic subvarieties of $\cH$, since it is the preimage of such a collection under the period map to $\bP(H^{1,0})$, which is the projectivization of the Hodge bundle over $\cA_g$.
 See \autoref{rmk:finiteness_in_A_g} for an explanation. 
 
 The discussion above, in particular \autoref{eqn:cM_subset_eigenform_locus}, gives that $\cM\subseteq \cN'$. 
 In the neighborhood $N_\epsilon$ of $(X_0,\omega_0)$ let $\cN$ be one of the irreducible components of $\cN'$ which contains $\cM$.
 We will check this component is unique, and coincides with $\cM$.
 
 Recall the defining conditions of $\cN'$ in $N_\epsilon$:
 \begin{align*}
 N_\epsilon\cap \cN' = \{(X,\omega)\vert \forall a \in k &\textrm{ we have } \rho(a) \textrm{ is of Hodge type }(0,0)\\
 &\textrm{ and } \rho(a)\omega = \iota_0(a)\omega\}
 \end{align*}
 The local systems $H^1_\iota$ are defined on $\cN$, but are extended to $N_\epsilon$ using the flat connection.
 They serve as ``eigen-systems'' for the action of $\rho(a)$, which itself is defined in $N_\epsilon$ using the flat connection.
 We then have the containment
 $$
 N_\epsilon \cap \cN' \subseteq \{(X,\omega)\; \vert\; \forall a\in k \textrm{ we have }\rho(a)\omega = \iota_0(a)\omega\}
 $$
 However, we also have the (local) equality which follows from the identifications via parallel transport
 $$
 \{(X,\omega)\vert \forall a\in k, \rho(a)\omega=\iota_0(a)\omega\} = \{(X,\omega)\vert p\left(\Pi(X,\omega)\right)\in H^1_{\iota_0, (X_0,\omega_0)} \}
 $$
 Looking back at the local definition of $\cM$ given by the inclusions \eqref{eqn:loc_def_M}, we see that this locus is exactly $\cM$.
 So we found that locally near $(X_0,\omega_0)$ we have
 $$
 \cM\subseteq \cN\subseteq \cM
 $$
 This finishes the proof that $\cM$ is an algebraic subvariety of $\cH$.
\end{proof}

\begin{remark}
\label{rmk:loc_syst_per_coord}
In local period coordinates, requiring the section $\omega:\cH\to T\cH$ to be in some local system is the same as restricting to a (local) affine manifold.
Algebraicity in the above theorem followed because we could identify the tangent space to $\cM$ with $p^{-1}(H^1_{\iota_0})$.
In general, we need to know the precise location of the tangent space in relative cohomology.
The next few sections deal with this question.
\end{remark}

\begin{remark}
 \label{rmk:finiteness_in_A_g}
 We now discuss the finiteness of the components of the eigenform locus for real multiplication.
 Assume the type of real multiplication is fixed, in other words the order $\cO$ and the isomorphism of $\cO$-lattice with symplectic form corresponding to the factor of the Jacobian.
 
 To prove finiteness of the eigenform locus, it suffices to prove finiteness of the real multiplication locus in $\cA_g$.
 Indeed, the eigenform locus is a projective space bundle over the real multiplication locus.
 
 Finiteness of the real multiplication locus will follow from the following general theorem of Borel and Harish-Chandra \cite{BHC}.
 Let $\Gamma$ be an arithmetic lattice in a $\bQ$-algebraic group $G$.
 Consider some representation $V$ of $G$, with a $\bZ$-structure compatible with $\Gamma$, and an integral vector $v\in V(\bZ)$ such that $G(\bR)\cdot v\subset V(\bR)$ is closed (equivalently, the stabilizer of $v$ is reductive).
 Then the set of integral points in the orbit $G(\bR)\cdot v$ form \emph{finitely many} classes under the action of $\Gamma$.
 
 Consider the decomposition of absolute cohomology given in \autoref{eqn:H1_decomp}, and abbreviate it as $H^1=M\oplus V$ where $M$ is the factor with real multiplication.
 Each factor contains a lattice, denoted $M(\bZ)$ and $V(\bZ)$ respectively.
 
 Now consider the abstract $\cO\oplus\bZ$-module $M\oplus V$, equipped with the compatible symplectic form (the $\bZ$-factor in $\cO\oplus\bZ$ acts on $V$ only).
 Associated to it is the period domain $\bH_{M,V}$ parametrizing pairs of abelian varieties $(A_M,A_V)$, with markings $M \xrightarrow{\sim} H^1(A_M)$ and $V\xrightarrow{\sim} H^1(A_V)$.
 Moreover, the markings should respect the symplectic forms and $\cO$ should act by endormorphisms of $A_M$.
 
 Consider possible embeddings $\phi: M(\bZ)\oplus V(\bZ) \into H^1_\bZ$, respecting the symplectic form.
 Choosing a basis of $M(\bZ)$ and $V(\bZ)$, such $\phi$ are determined by a collection of vectors in $H^1_\bZ$ subject to constraints coming from the symplectic pairing between the vectors.
 Equivalently, this is determined by a single vector in the direct sum of several $H^1_\bZ$'s, with the appropriate constraints.
 
 The set of real vectors subject to the same constraints forms a single orbit under $\Sp(H^1)(\bR)$.
 Therefore, by the Borel--Harish-Chandra theorem there are finitely many possible $\phi$ up to the action of $\Sp(H^1)(\bZ)$ on the target.
 
 Finally, each $\phi$ determines an embedding $I_\phi$ of the period domain $\bH_{M,V}$ into the Siegel space corresponding to $H^1$.
 The stabilizer of the image of $I_\phi$ acts with co-finite volume on this image.
 Since there are only finitely many $\Sp(H^1)(\bZ)$-equivalence classes of $\phi$'s, there are finitely many corresponding subvarieties in $\cA_g$.
 These subvarieties of $\cA_g$ parametrize abelian varieties with real multiplication by $\cO$ on a factor and $\cO$-module structure as the one corresponding to the affine manifold $\cM$.
\end{remark}

\section{Mixed Hodge structures and their splittings}
\label{sec:MHS}

This section contains background on mixed Hodge structures and their properties.
The monograph of Peters and Steenbrink \cite{PetersSteenbrink} provides a thorough treatment.
We include examples relevant to our situation.
The full machinery, as developed e.g. by Carlson in \cite{Carlson}, is not strictly necessary.
However, using this language streamlines some of the arguments.
Throughout this section, we fix a ring $k$ such that $\bZ\subseteq k\subseteq \bR$.
Most often, $k$ will be a field.

\subsection{Definitions}

First recall some standard definitions.

\begin{definition}
 A $k$-Hodge structure of weight $w$ is a $k$-module $V_k$ and the data of a decreasing filtration by complex subspaces $F^\bullet$ on $V_\bC:=V_k\otimes_k \bC$
 $$
 \cdots\subseteq F^p\subseteq F^{p-1}\subseteq\cdots\subseteq V_\bC
 $$
 The filtration is called the Hodge filtration and is required to satisfy
 $$
 F^p\oplus \conj{F^{w+1-p}}=V_\bC
 $$ 
\end{definition}

\begin{definition}
 A $k$-Mixed Hodge structure is a $k$-module $V_k$ with an increasing filtration $W_\bullet$ defined over $k\otimes_\bZ \bQ$
 $$
 \cdots W_n\subseteq W_{n+1}\subseteq \cdots \subseteq \left(V_k\otimes_\bZ \bQ\right)
 $$
 and a decreasing filtration $F^\bullet$ on $V_\bC$ such that $F^\bullet$ induces a $k$-Hodge structure of weight $n$ on
 $$
 \gr_n^{W} V = W_n/W_{n-1}
 $$
 The filtration $W_\bullet$ is called the weight filtration.
\end{definition}

\begin{remark}
 We can take duals of (mixed) Hodge structures and overall, they form an abelian category. 
 Negative indexing in the filtration is allowed.
 See \cite[Section 3.1]{PetersSteenbrink} for more background.
 
 For future use, we recall the definition of the dual Hodge structure.
 To describe the Hodge and weight filtrations on the dual of $V$, denoted $V^\vee$, let 
 \begin{equation} 
  \begin{aligned}
  \label{eqn:def_dual} 
    F^{p+1} V^\vee &= \{\xi\in V^\vee\vert \xi(F^{-p} V)=0\}\\
    W_{n-1} V^\vee &= \{\xi \in V^\vee\vert \xi(W_{-n})=0\}
  \end{aligned}
 \end{equation}
\end{remark}
 
\begin{remark}
 In the definition of mixed Hodge structures, the weight filtration was defined only after allowing $\bQ$-coefficients.
 However, if it came from a $\bZ$-module, the position of the lattice will be relevant.
\end{remark}

\begin{example}
\label{eg:mhs01}
 Let $C$ be a compact Riemann surface and $S\subset C$ a finite set of points.
 On the relative cohomology group $H^1(C,S)$ we have a natural mixed Hodge structure with weights $0$ and $1$.
 This is the same as the compactly supported cohomology of the punctured surface $C\setminus S$.
 
 We have the exact sequence
 $$
 0\to W_0\into H^1(C,S)\onto H^1(C)\to 0
 $$
 The sequence is valid with any coefficients, so we consider it over $\bZ$.
 We have the canonical identification $W_0=\widetilde{H^0(S)}$ which is the reduced cohomology of the set $S$.
 
 Here is the mixed Hodge structure on $H^1(C,S)$.
 The weight filtration has $W_0$ defined by the exact sequence, and $W_1$ is the entire space.
 The holomorphic $1$-forms on the Riemann surface $C$ give also relative cohomology classes, so form a subspace
 $$
 F^1\subset H^1_\bC(C,S)
 $$
 This subspace maps isomorphically onto the holomorphic $1$-forms on $H^1_\bC(C)$.
 This describes the mixed Hodge structure, and according to Carlson \cite[Theorem A]{Carlson} it recovers the punctured curve in many cases. 

\paragraph{The dual picture.}
We shall often work with duals, because the constructions are more natural.
Dualizing the above sequence we find
$$
0\from W_0^\vee \from \check{H}^1(C,S)\from \check{H}^1(C)\from 0
$$
On the space $\check{H}^1(C,S)$ we have a mixed Hodge structure of weights $-1$ and $0$ (see \autoref{eqn:def_dual}).

The space $W_{-1}$ is the image of $\check{H}^1(C)$ which is pure of weight $-1$.
The space $W_0$ is everything.
The Hodge filtration has $F^0\check{H}^1(C,S)$ equal to the annihilator of $F^1H^1(C,S)$.
In particular it contains the image of $F^0\check{H}^1(C)$, which is the annihilator of $F^1H^1(C)$.
Moreover, we have the (natural) isomorphism over $\bC$
$$
W_0^\vee\from F^0\check{H}^1(C,S)/F^0\check{H}^1(C)
$$
\end{example}

\subsection{Splittings}

The concepts below were first analyzed by Carlson \cite{Carlson}, which provides more details.
We work exclusively with mixed Hodge structures of weights $\{0,1\}$ and their duals, with weights $\{-1,0\}$.
They are viewed as extensions of pure Hodge structures of corresponding weights.
\autoref{eg:mhs01} describes the mixed Hodge structures that occur in later sections.

\begin{definition}
\label{def:split}
 Fix a ring $L$ with $k\subseteq L \subseteq \bR$.
 A $k$-mixed Hodge structure
 $$
 0\to W_0 \to E \to H^1\to 0
 $$
 is \emph{$L$-split} if the sequence, after extending scalars to $L$, is isomorphic as a sequence of $L$-mixed Hodge structures to
 $$
 0\to W_0 \to W_0\oplus H^1\to H^1\to 0
 $$
 The mixed Hodge structure in this sequence is the direct sum of the pure structures.
 The isomorphism is required to be defined over $L$, but it must map the weight and Hodge filtrations isomorphically.
\end{definition}

\begin{remark}
\label{remark:splittings}
\leavevmode
 \begin{enumerate}
  \item[(i)] The definition for splittings of duals is analogous.
  A mixed Hodge structure is $L$-split if and only if its dual is.
  \item[(ii)] Giving a splitting is the same as giving a map defined over $L$
  $$
  \sigma: H^1\to E
  $$
  which is the identity when composed with projection back to $H^1$.
  It is required to map $F^1H^1$ isomorphically to $F^1E$.
  \item[(iii)] Mixed Hodge structures as above are always $\bR$-split.
  In the dual picture, we have the sequence
  $$
  0\from W_0^\vee \from E^\vee \from \check{H}^1\from 0
  $$
  Then $F^0E^\vee\cap \conj{F^0E^\vee}$ is a real subspace which maps isomorphically onto $W_0^\vee (\bR)$.
  Over $\bR$ we can thus lift $W_0^\vee$ inside $E^\vee$ using this subspace and this provides the splitting.  
 \end{enumerate}
\end{remark}

\begin{example}
\label{eg:splitting_ell_curve}
 It is more convenient to describe splittings of dual sequences, and below is the simplest example.
 Consider a pure weight $-1$ Hodge structure $\check{H}^1_\bZ=\langle a,b\rangle$ with filtration $F^0\check{H}^1=\langle a+\tau b\rangle$, where $\Im \tau>0$.
 This defines an elliptic curve 
 \[
  \Jac(\check{H}^1):=\check{H}^1_\bC/\left(F^0\check{H}^1 + \check{H}^1_\bZ \right) \cong \raisebox{0.1em}{$\bC$} /  \raisebox{-0.1em} {$(\bZ \oplus\bZ\tau)$}
 \]
 Note that $\check{H}^1_\bC=F^0\check{H}^1 \oplus \conj{F^0\check{H}^1}$ (since $\tau\neq \conj{\tau}$) and in particular $\check{H}^1_\bZ\cap F^0\check{H}^1=\{0\}$ (since integral elements are invariant under complex conjugation).
 
 Consider now possible extensions of the form
 $$
 0\from W_0\from E \from \check{H}^1\from 0
 $$
 Assume $W_0$ is of $\bZ$-rank $1$, generated by $c$.
 Lift it to some $c_1\in E_\bZ$.
 It gives a map
 $$
 \sigma_\bZ:W_0(\bZ) \to E_\bZ
 $$
 Note that the lift $c_1$ is ambiguous, up to addition of terms $xa+yb$ with $x,y\in \bZ$.
 Here, the generators of $\check{H}^1$ are identified with their image inside $E$.
 
 The extra data on $E$ is a subspace $F^0E$ which contains $F^0\check{H}^1$, is complex two-dimensional, and maps surjectively onto $W_0$.
 Pick a vector $v\in F^0E$ which maps to $c$.
 It must be of the form
 $$
 v=c_1 + \lambda a + \mu b
 $$
 where $\lambda,\mu \in \bC$.
 Note that the lift $v$ is ambiguous, up to the addition of complex multiples of $a+\tau b$ (which generate $F^0\check{H}^1=F^0E\cap \ker (E\to W_0) $).
 This provides a second lift
 $$
 \sigma_\bR:W_0(\bZ)\to E_\bC/F^0 \check{H}^1
 $$
 We can take the difference of $\sigma_\bZ$ and $\sigma_\bR$.
 Because projecting $\sigma_\bZ-\sigma_\bR$ back to $W_0$ is the zero map, their image must land in $\check{H}^1_\bC$.
 Taking into account also the ambiguity in the definition of $\sigma_\bZ$, we get a \emph{canonical} map
 $$
 \sigma_\bZ-\sigma_\bR:W_0(\bZ)\to \check{H}^1_\bC/\left(F^0\check{H}^1 + \check{H}^1_\bZ \right)
 $$
 So we get a canonical element of the elliptic curve associated to the Hodge structure $\check{H}^1$. 
 This element is zero if and only if the sequence is $\bZ$-split.
 Indeed, the vanishing of this element means we could choose the lift $v$ above to be integral.
 
 The element is torsion in the elliptic curve if and only if the sequence is $\bQ$-split.
 It means we could choose $v$ above with rational coefficients in $a$ and $b$.
\end{example}

\subsection{Extension classes and field changes}
\label{subsec:exts_fields}

This section contains a discussion of algebraic facts needed later.
The details of the constructions are available in \cite{Carlson} and \cite[Chapter 3.5]{PetersSteenbrink}.

\paragraph{Jacobians and extensions}
For a $\bZ$-Hodge structure $H^1$ of weight $1$, its Jacobian is defined using the dual Hodge structure $\check{H}^1$ by
$$
\Jac_\bZ H^1 := \check{H}^1(\bC)/\left(F^0\check{H}^1+\check{H}^1(\bZ)\right)
$$
This is a compact complex torus, again since $\check{H}^1_\bC=F^0\check{H}^1 \oplus \conj{F^0\check{H}^1}$ and the $\bZ$-lattice doesn't intersect $F^0$.
As in \autoref{eg:splitting_ell_curve}, extensions of $H^1$ by a weight $0$ Hodge structure $W_0$
$$
0\to W_0 \to E \to H^1\to 0
$$
are classified by elements in $\Hom_\bZ(\check{W}_0(\bZ),\Jac_\bZ H^1)$.
Rather, it is dual extensions that are classified by such elements.

Let now $K$ be a larger field or ring $\bZ\subseteq K\subseteq \bR$ and $H^1$ be a $K$-Hodge structure of weight $1$.
We can also define a Jacobian
$$
\Jac_K H^1 := \check{H}^1(\bC)/\left(F^0\check{H}^1+\check{H}^1(K)\right)
$$
It is an abelian group (even a $K$-vector space, usually of infinite dimension) with no structure of manifold.
Extensions are still classified by elements of $\Hom_K(\check{W}_0(K),\allowbreak \Jac_K H^1)$, where the Jacobian is viewed as a $K$-vector space.

\paragraph{Morphisms}
If $H^1$ has a $\bZ$-structure and we extend scalars to $K$, then we have a natural map of abelian groups
$$
\Jac_\bZ H^1\to \Jac_K H^1
$$
For example $\Jac_\bQ H^1=\Jac_\bZ H^1/\langle\textrm{torsion}\rangle$ and an extension class is torsion in the usual Jacobian if and only if the extension splits over $\bQ$.

More generally, suppose we have (after extending scalars to $K$) an inclusion of Hodge structures $H^1_\iota \into H^1$.
The dual map becomes $\check{H}^1\onto \check{H}^1_\iota$ and we have an induced surjection on Jacobians
$$
\Jac_\bZ H^1 \onto \Jac_K H^1_\iota
$$
From an extension class $\xi\in \Hom_\bZ(\check{W}_0(\bZ),\Jac_\bZ H^1)$ we get another one $\xi_\iota\in \Hom_K(\check{W}_0(K),\allowbreak \Jac_K H^1)$ by composing with the above map.
This gives a corresponding extension of $K$-Hodge structures.
Given a $K$-subspace $S\subseteq \check{W}_0(K)$ we can restrict the extension class $\xi_\iota$ to it and get another such extension.
Note that a subspace of $\check{W}_0(K)$ corresponds to a quotient of $W_0(K)$.

\section{Splittings over Affine Invariant Manifolds}
\label{sec:splittings}

In this section we identify variations of mixed Hodge structure that naturally exist on affine invariant manifolds.
The main result is that they split after an appropriate extension of scalars, in the sense of the previous section.

\subsection{Setup.}

Consider an affine invariant manifold $\cM$ inside a stratum of flat surfaces $\cH$.
It was proved in \cite[Theorem 7.3]{ssimple} that the variation of Hodge structure over $\cM$ given by the first cohomology $H^1$ splits as
$$
H^1=\left(H^1_{\iota_0}\oplus\cdots\oplus H^1_{\iota_{r-1}}\right)\oplus V
$$
Each term above gives a variation of Hodge structure.
Moreover, the direct sum in the parenthesis comes from a $\bQ$-local system. 
Each summand $H^1_\iota$ corresponds to an embedding $\iota$ of a totally real number field $k$, and comes from a local system defined over that embedding.
The embedding $\iota_0$ is the distinguished embedding.

Recall the tangent space $T\cH$ to the stratum is given, via period coordinates, by the relative cohomology of the underlying surfaces.
Restrict the tangent bundle $T\cH$ to $\cM$.
It projects to $H^1$ and we can take the preimage of the summands coming from the totally real field.
This yields an exact sequence of $\bQ$-local systems
\begin{align}
\label{eqn:big_ses}
0\to W_0 \to E \overset{p}{\rightarrow} \left(H^1_{\iota_0}\oplus\cdots \oplus H^1_{\iota_{r-1}}\right) \to 0
\end{align}
Here $W_0$ is the local system of purely relative cohomology classes.
It coincides with the (reduced) cohomology of the marked points (i.e. zeroes of the $1$-form).

This provides a variation of $\bQ$-mixed Hodge structures in the following sense.
Above each point in $\cM$ we have an induced mixed Hodge structure.
The Hodge filtrations $F^\bullet$ give holomorphic subbundles of the corresponding local systems.
The Griffiths transversality conditions are empty in this case.

\begin{remark}
 Below we discuss local systems defined over a particular field, for example $\iota_0(k)$ from above.
 To define this notion, fix a normal closure $K$ containing all the embeddings of $k$, with Galois group over $\bQ$ denoted $G$.
 Given a $\bQ$-local system $V_\bQ$, we can extend scalars to $K$ and denote it $V_K$.
 A sublocal system $W\subset V_K$ is ``defined over $\iota_l(k)$" if it is invariant by the subgroup $G_{\iota_l}\subset G$ stabilizing $\iota_l(k)$.
 
 We will omit the explicit extension of scalars to $K$ below, and just say that $W\subset V$ is a sublocal system defined over $\iota_l(k)$.
\end{remark}

\paragraph{The tangent space of the affine manifold.}
According to results of Wright \cite[Thm. 1.5]{Wright_field} the tangent space $T\cM$ to the affine submanifold gives a local subsystem $T\cM\subset E$, which is defined over $k$ (rather, $\iota_0(k)$).
It has the property that $p(T\cM)=H^1_{\iota_0}$.
Define the kernel of the map $p$ (see \autoref{eqn:big_ses}) by
$$
W_0\cM:=W_0\cap T\cM
$$
Define also $T\cM_{\iota_l}$ to be the Galois-conjugate of the local system $T\cM$ corresponding to the embedding $\iota_l$.
Analogously define
$$
W_0\cM_{\iota_l}:=W_0\cap T\cM_{\iota_l}
$$
Note that $T\cM_{\iota_l}$ surjects onto $H^1_{\iota_l}$ with kernel $W_0 \cM_{\iota_l}$.

\subsection{The splitting}

The space $H^1_{\iota_l}$ contains holomorphic $1$-forms denoted $H^{1,0}_{\iota_l}$.
These also give \emph{relative} cohomology classes, i.e. a natural map $H^{1,0}_{\iota_l}\to H^1_{rel}$.
The main theorem of this section (below) is the compatibility of this map with $T\cM_{\iota_l}$. 
Note that in the case when $W_0$ is contained in $T\cM$ (i.e. the setup of \autoref{sec:alg_special_case}) the statement below holds trivially.

\begin{theorem}
\label{thm:iota_l_splitting}
 Consider the variation of $\iota_l(k)$-mixed Hodge structures
 \begin{align}
 \label{eqn:ses_iota_l}
 0\to W_0/W_0\cM_{\iota_l} \into p^{-1}(H^1_{\iota_l})/W_0\cM_{\iota_l} \onto H^1_{\iota_l}\to 0
 \end{align}
 It is an exact sequence of local systems defined over $\iota_l(k)$, and each space carries compatible (mixed) Hodge structures.
 
 Then this sequence is (pointwise) split over $\iota_l(k)$ in the sense of \autoref{def:split}.
 
 The splitting is provided by the local system $T\cM_{\iota_l}/W_0\cM_{\iota_l}$, i.e. the surjection in the sequence \eqref{eqn:ses_iota_l} can be split by an $\iota_l(k)$-isomorphism
 \begin{align}
 \label{eqn:H^1_inverse}
 H^1_{\iota_l} \toisom T\cM_{\iota_l}/W_0\cM_{\iota_l}
 \end{align}
\end{theorem}
Note that there are two ways to lift a cohomology class in $H^{1,0}_{\iota_l}$ to $H^1_{rel}$.
One uses the map \eqref{eqn:H^1_inverse}, and the other is by viewing a holomorphic $1$-form as a relative cohomology class.
This theorem claims that these two ways in fact agree.
\begin{proof}
The proof is in three steps.
In Step 1 we dualize the sequence \eqref{eqn:ses_iota_l} and use the Galois-conjugate tangent space to produce a flat splitting of the local systems.
We also find the $\bR$-splitting coming from the underlying Hodge structures.
Their difference is a (holomorphic) section of a bundle with negative curvature.

In Step 2, we show the section must have constant Hodge norm.
In Step 3 we use this to show that the section must come from a flat one, and thus must in fact be zero.
This concludes the proof, since it shows that the $\bR$-splitting was in fact defined over $\iota_l(k)$.
The next three sections deal with each step.
\end{proof}

\subsection{Proof of Step 1}

Because in the exact sequence \eqref{eqn:ses_iota_l} we quotient out $W_0\cM\iota_l$  we deduce the bundle $T\cM_{\iota_l}/W_0\cM_{\iota_l}$ maps isomorphically onto $H^1_{\iota_l}$.
Dualizing the sequence \eqref{eqn:ses_iota_l} we obtain
\begin{align}
\label{eqn:ses_iota_l_dual}
0\from \left(W_0/W_0\cM_{\iota_l}\right)^\vee \twoheadleftarrow \left(p^{-1}\left(H^1_{\iota_l}\right)/W_0\cM_{\iota_l}\right)^\vee \hookleftarrow \check{H}^1_{\iota_l}\from 0
\end{align}
Denote the annihilator of $T\cM_{\iota_l}/W_0\cM_{\iota_l}$ by $\left(T\cM_{\iota_l}/W_0\cM_{\iota_l}\right)^\perp$.
By the remark above, it maps isomorphically onto $\left(W_0/W_0\cM_{\iota_l}\right)^\vee$.
The inverse isomorphism defines a canonical flat map of $\iota_l(k)$ local systems, which is a splitting of the left surjection in the exact sequence \eqref{eqn:ses_iota_l_dual}
$$
\sigma_{\iota_l}:\left(W_0/W_0\cM_{\iota_l}\right)^\vee \to \left(p^{-1}\left(H^1_{\iota_l}\right)/W_0\cM_{\iota_l}\right)^\vee
$$

We now construct another splitting using the Hodge structures (see \autoref{remark:splittings} (iii)).
Consider the $F^0$ piece of the middle term in the sequence \eqref{eqn:ses_iota_l_dual}.
It maps surjectively onto $\left(W_0/W_0\cM_{\iota_l}\right)^\vee$ with kernel $F^0\check{H}^1_{\iota_l}$.
This gives a canonical splitting
$$
\sigma_\bR: \left(W_0/W_0\cM_{\iota_l}\right)^\vee_\bC\to \left(p^{-1}\left(H^1_{\iota_l}\right)/W_0\cM_{\iota_l}\right)^\vee_\bC/F^0\check{H}^1_{\iota_l}
$$
Note that because it really comes from an isomorphism of vector bundles, it is in fact holomorphic over the affine manifold.

Note that both maps $\sigma_\bR$ and $\sigma_{\iota_l}$ are splittings. 
This means that composing either with the surjection onto the left term of the sequence \eqref{eqn:ses_iota_l_dual} gives the identity.
So their difference has image in the kernel of the surjection in \eqref{eqn:ses_iota_l_dual}:
$$
\sigma_{\iota_l}-\sigma_\bR: \left(W_0/W_0\cM_{\iota_l}\right)^\vee_\bC\to \check{H}^1_{\iota_l}/F^0\check{H}^1_{\iota_l}
$$
Next, we assume that we are in some finite cover of the stratum where the local system $W_0$ is trivial; labeling the zeroes of the $1$-form suffices.
Then we can choose an element (same as a global section) of $\left(W_0/W_0 \cM_{\iota_l}\right)^\vee_\bC$ denoted by $e$.
By taking its image under the above map, we obtain a global over $\cM$ holomorphic section
$$
\psi_e :=(\sigma_{\iota_l}-\sigma_{\bR})(e)\in \Gamma(\cM; \check{H}^1_{\iota_l}/F^0\check{H}^1_{\iota_l})
$$

\subsection{Proof of Step 2}

Given the holomorphic section $\psi_e$ produced in Step 1, we now show it has constant Hodge norm.

\paragraph{Notation.} Denote the Hodge decomposition of $\check{H}^1$ by
$$
\check{H}^{1} = \check{H}^{0,-1}\oplus \check{H}^{-1,0}
$$
We keep the same notation for indices involving the embeddings $\iota_l$.

Note that $H^{0,-1}=F^0\check{H}^1$ gives a holomorphic subbundle, being identified with the annihilator of $H^{1,0}=F^{1}{H}^1$ (which is the bundle of holomorphic $1$-forms).
Using the polarization of $H^1$, we see that $\check{H}^1$ is isomorphic to $H^1$, up  to a shift of weight by $(1,1)$.

The section $\psi_e$ produced above is a holomorphic section of $\check{H}^{-1,0}$, endowed with the holomorphic structure when viewed as a quotient $\check{H}^{-1,0} = \check{H}^1/F^0\check{H}^1$.
This bundle has negative curvature (see \cite[Corollary 3.15]{ssimple}, with a weight shift).

We want to apply Lemma 5.2 from \cite{ssimple} to conclude the function $\log \norm{\psi_e}$ is constant.
For this, we need to check the boundedness and sublinear growth assumptions.
This is done below.

Note that this function is subharmonic by the calculation in \cite[Lemma 3.1]{ssimple} and the negative curvature of the bundle.
If $\norm{\psi_e}\neq 0$ identically, then it can vanish only on a lower-dimensional analytic subset.
Therefore, one can define the functions $f_N:=\max(-N,\log \norm{\psi_e})$ and apply \cite[Lemma 5.2]{ssimple} to them (and let $N\to +\infty$).
As the maximum of two subharmonic functions, each $f_N$ is itself subharmonic.

\paragraph{Checking assumptions.}
First, we examine how $\psi_e$ was defined.
We have the exact sequence
\begin{align}
\label{ses:all}
0\to W_0 \to E \to (H^1_{\iota_0}\oplus\cdots \oplus H^1_{\iota_{r-1}}) \to 0
\end{align}
and its dual
$$
0\from W^\vee_0 \from E^\vee \from (\check{H}^1_{\iota_0}\oplus\cdots\oplus \check{H}^1_{\iota_{r-1}} )\from 0
$$
We have the $\bR$-splitting coming from Hodge theory (see \autoref{remark:splittings} (iii))
$$
W_0^\vee(\bR)\to E^\vee_{\bR}
$$
This gives a direct sum decomposition
\begin{align}
\label{eqn:mhs_splitting}
E^\vee(\bR) \cong W_0^\vee(\bR)\oplus \check{H}^1(\bR)
\end{align}
and an induced metric from each factor.
At the level of the exact sequence \eqref{ses:all} this is the same as using the harmonic representatives of absolute cohomology classes to obtain relative cohomology classes.

\paragraph{Terminology for norms.}
Using the splitting from \eqref{eqn:mhs_splitting} and more generally the same for $H^1_{rel}$, we can put norms on the cocycle by putting a norm on each factor.
Several norms will appear below, and we explain now the terminology.

On $H^1$ and its dual $\check{H}^1$ one has the Hodge norms, coming from the Hodge structure (which are preserved by duality).
These Hodge norms will be denoted $\norm{-}$ in the sequel.
Eskin, Mirzakhani and Mohammadi \cite[Sec. 7]{EMM} introduced a modified Hodge norm on $H^1$, obtained by increasing the usual Hodge norm.
Note that the dual of this modified Hodge norm, on $\check{H}^1$, is \emph{less than} the Hodge norm on $\check{H}^1$.
The modified Hodge norms on $H^1$ and $\check{H}^1$ will be denoted $\norm{-}'$.

Finally, we have the modified mixed Hodge norms on $H^1_{rel}$ and $\check{H}^1_{rel}$.
These will be denoted $\vertiii{-}$ (note that they will not be dual to each other).
On $H^1_{rel}$, the modified mixed Hodge norm is defined in \cite{EMM} by putting the constant norm on $W_0$ and the modified Hodge norm on $H^1$.
On $\check{H}^1_{rel}$, the modified mixed Hodge norm is defined in \autoref{eqn:dual_modified_Hg} by changing the norm on the $W_0^\vee$ factor.

Below, the adjective \emph{dual} means that the norm is on the dual space, and the adjective \emph{mixed} means that it is in relative (co)homology.
We now proceed to the details.

\paragraph{Modified Hodge norm.}
The modified Hodge norm $\norm{-}'$ on $H^1$ is defined in \cite[eqn. (33) and below]{EMM}) and the cocycle on $H^1$ is integrable for it (\cite[Lemma 7.4]{EMM}).

Next, recall the splitting $H^1_{rel}=W_0\oplus H^1$ coming from Hodge theory using holomorphic lifts (in their language -- harmonic representatives).
Using it, the modified Hodge norm $\vertiii{-}$ on $H^1_{rel}$ is defined using the modified Hodge norm on $H^1$ and the constant norm on $W_0$ (see \cite[eq. (40) and above]{EMM}).

For this modified Hodge norm on $H^1_{rel}$ the cocycle is bounded \cite[Lemma 7.5]{EMM}.
The main consequence \cite[Lemma 7.6]{EMM} is the following.
For the splitting $H^1_{rel}=W_0 \oplus H^1$, write the cocycle matrix as
\begin{align}
\label{eqn:explicit_KZ}
 \begin{bmatrix}
  1 & U(x,s)\\
  0 & A(x,s)
 \end{bmatrix}
 \textrm{ where }U(x,s):H^1_x \to (W_0)_{g_s x}
\end{align}
Then $\norm{U(x,s)}_{op}\leq e^{m's}$ for some $m'$, where $H^1$ is viewed with the modified Hodge norm $\norm{-}'$ and $W_0$ with the constant norm.
Explictly, this says
\begin{align}
\label{eqn:U_bdd_H1}
 \norm{U(x,s)v}\leq e^{m's}\norm{v}' \textrm{ where }v\in H^1_{x}\textrm{ and so }U(x,s)v\in (W_0)_{g_sx}
\end{align}

\paragraph{Dual modified Hodge norm.}
For a point $x$ in the stratum, let $r(x)\in [1,\infty)$ be the ratio of the modified to the usual Hodge norms on $H^1$.
In other words
\begin{align}
 r(x) := \sup_{0\neq v\in H^1_x} \frac{\norm{v}' }{\norm{v} } 
\end{align}
Note that $r(x)\geq 1$ since $\norm{-}'$ is always at least the usual Hodge norm (see before eq. (40) in \cite{EMM}).
By definition $r(x)\cdot \norm{v} \geq \norm{v}'$ for any $v\in H^1_x$.

On $\check{H}^1$ there are two norms -- the dual of the usual Hodge norm, denoted $\norm{-}$, and the dual of the modified Hodge norm, denoted $\norm{-}'$.
The inequality involving $r(x)$ is now reversed, i.e. letting $\xi\in \check{H}^1_x$, we have
\begin{align}
\label{eqn:r_dual_bound}
 \norm{\xi} \leq r(x)\cdot \norm{\xi}'
\end{align}
To see this, using $\frac 1 {r(x)} \norm{v}'\leq \norm{v}$ for the second step, we have
\begin{align}
 \norm{\xi}  = \sup_{\norm{v}=1} |\xi(v)| \leq \sup_{\frac 1 r \norm{v}'=1} |\xi(v)| = r(x)\norm{\xi}'
\end{align}

\paragraph{Dual modified mixed Hodge norm.}
We now explain how to modify the Hodge norm on the dual cocycle $\check{H}^1_{rel}$ which is relevant to our arguments.
Recall the splitting $\check{H}^1_{rel} = W_0^\vee \oplus \check{H}^1$ coming from Hodge theory.
Given a vector $w\oplus h\in H^1_{rel}$ its ordinary mixed Hodge norm is $\norm{w\oplus h}^2=\norm{w_0}^2+\norm{h}^2$ (using the usual Hodge norm for $h$ and constant norm for $w_0$).
Define now its dual modified mixed Hodge norm via
\begin{align}
\label{eqn:dual_modified_Hg}
 \vertiii{w\oplus h}^2 := \norm{w_0}^2\cdot r(x)^2 + \norm{h}^2
\end{align}
Note that the modification affects only the $W_0^\vee$ part of the norm, not the $\check{H}^1$.
\begin{proposition}
 For the modified dual Hodge norm defined in \autoref{eqn:dual_modified_Hg}, the Kontsevich-Zorich cocycle is bounded.
\end{proposition}
\begin{proof}
 Given the explicit form of the KZ cocycle in \autoref{eqn:explicit_KZ}, its dual cocycle (for the dual splitting) will be given by the inverse transpose of that matrix.
 This reads
 \begin{align}
  \left(
  \begin{bmatrix}
  1 & U(x,s)\\
  0 & A(x,s)
 \end{bmatrix}^t\right)^{-1} =
  \begin{bmatrix}
  1 & 0\\
  -(A(x,s)^t)^{-1} \cdot U(x,s)^t & (A(x,s)^t)^{-1}
 \end{bmatrix}
 \end{align}
 Next, recall that the cocycle $A(x,s)$ is bounded in both forward and backwards time, since it corresponds to $H^1$, and we have the usual Hodge norm on that part.
 To show boundedness of the full cocycle, it suffices to prove that $\vertiii{U(x,s)^t}_{op}\leq e^{m's}$ for some $m'$.
 Note that $U(x,s)^t : (W_0^\vee)_{g_sx}\to \check{H}^1_x$, since $U(x,s)$ goes the other way between dual spaces (see \eqref{eqn:explicit_KZ}).
 The bound on the operator norm of $U(x,s)^t$ needs to hold when $W_0^\vee$ is viewed with the norm $r(x)\norm{-}$, and $\check{H}^1$ is with the usual Hodge norm.
 
 The operator $U(x,s)$ is bounded for the modified Hodge norm on $H^1$ and usual norm on $W_0$ (see \eqref{eqn:U_bdd_H1}).
 Therefore $U(x,s)^t$ satisfies the same bound for the dual modified Hodge norm on $\check{H}^1$ and usual norm on $W_0^\vee$.
 This reads
 \begin{align}
  \norm{U(x,s)^t \check{w}}' \leq e^{m's} \norm{\check{w}} \textrm{ where }\check{w}\in (W_0^\vee)_x 
 \end{align}
 However, using \autoref{eqn:r_dual_bound} which relates the usual and modified Hodge norms in the dual $\check{H}^1$ we find
 \begin{align}
  \norm{U(x,s)^t \check{w}}\cdot \frac 1 {r(x)} \leq \norm{U(x,s)^t \check{w}}' \leq e^{m's} \norm{\check{w}}
 \end{align}
 Moving $r(x)$ to the right side, this exactly says that $U(x,s)^t$ is bounded when $W_0^\vee$ carries the dual modified mixed Hodge norm.
 
 Finally, note that the identity map on $W_0^\vee$ no longer acts by isometries, since the norm on that factor depends now on the basepoint.
 However, from the boundedness of the cocycle on $H^1$ with the usual, as well as the modified Hodge norm, it follows that $r(g_s x)\leq e^{m''s}r(x)$ for all $s$.
 Thus, the Kontsevich-Zorich cocycle on $\check{H}^1_{rel}$ is bounded for the dual modified mixed Hodge norm.
\end{proof}

\paragraph{Properties of the dual modified mixed Hodge norm.}
First, note that the boundedness properties of the cocycle descend to the various pieces of $\check{H}^1$ and $\check{H}^1_{rel}$.
Returning to the notation in \autoref{ses:all}, \eqref{eqn:mhs_splitting}, we constructed a norm $\vertiii{-}$ satisfying
\begin{enumerate}
 \item [(i)] The Kontsevich-Zorich cocycle on $E$ and $E^\vee$ is integrable for this norm.
 Moreover, it satisfies the absolute bound for some universal constant $C>0$
 \begin{align}
 \label{eqn:g_t_bdd_mod_norm}
 -C\cdot T \leq \log \vertiii{g_T} \leq C\cdot T
 \end{align}
 \item [(ii)] The projection $E\to H^1$ is norm non-increasing and dually the embedding $\check{H}^1\into E^\vee$ is norm non-decreasing (where $H^1$ and $\check{H}^1$ have the usual Hodge norms).
 In other words, if $\psi\in \check{H}^1$ is a section, then 
 \begin{align}
 \label{eqn:norm_non_dec}
  \norm{\psi}\leq \vertiii{\psi} \textrm{\hskip 0.5em (in fact, we have equality in this case)}
 \end{align}
 Moreover, if $\phi$ if a section of $H^1_{rel}$ and $\psi$ is its $\check{H}^1$-component, then $\vertiii{\phi}\geq \norm{\psi}$.
 In other words, the dual modified mixed Hodge norm of $\phi$ dominates the usual Hodge norm of $\psi$, since the decomposition $\check{H}^1_{rel}=W_0^\vee \oplus\check{H}^1$ is orthogonal for the dual modified mixed Hodge norm.
\end{enumerate}
Consider now 
$$
0\to W_0 \to p^{-1}(H^1_{\iota_l})\to H^1_{\iota_l} \to 0
$$
In the dual picture, with the dual modified mixed Hodge norm, we have
\begin{align}
\label{eqn:ses_iota_l_dual_again}
0\from W_0^\vee \twoheadleftarrow p^{-1}(H^1_{\iota_l})^\vee \hookleftarrow \check{H}^1_{\iota_l} \from 0
\end{align}
Corresponding to $W_0\cM_{\iota_l}$ in the first sequence is its annihilator $W_0\cM_{\iota_l}^\perp$ in the dual.
It is naturally identified with $\left(W_0/W_0\cM_{\iota_l}\right)^\vee$.

Now $p^{-1}(H^1_{\iota_l})^\vee$ also contains the annihilator of $T\cM_{\iota_l}$ denoted $T\cM_{\iota_l}^\perp$.
The surjection in the dual sequence \eqref{eqn:ses_iota_l_dual_again} gives an isomorphism
$$
W_0\cM_{\iota_l}^\perp \from T\cM_{\iota_l}^\perp
$$
The section $\sigma_{\iota_l}$ from Step 1 is the inverse of this isomorphism.

Given $e\in(W_0/W_0\cM_{\iota_l})^\vee$ (which is naturally $W_0\cM_{\iota_l}^\perp$) we have $\phi_e$ defined as
$$
\phi_e:=\sigma_{\iota_l}(e) \in T\cM^\perp_{\iota_l}\subset p^{-1}(H^1_{\iota_l})^\vee
$$
Note that $\phi_e$ is a flat global section of $p^{-1}(H^1_{\iota_l})^\vee$.
This last bundle, equipped with the dual modified mixed Hodge norm, gives rise to an integrable cocycle; the Oseledets theorem thus holds.

Just like in \cite[Lemma 5.1]{ssimple} $\phi_e$ must be in the zero Lyapunov subspace, otherwise its norm would be unbounded on any set of positive measure.
In particular, its dual modified mixed Hodge norm grows subexponentially along a.e. \Teichmuller geodesic.

Recall the  splitting over $\bR$ for the sequence \eqref{eqn:ses_iota_l_dual_again}.
This comes from the decomposition
\begin{align}
 \label{eqn:dir_sum_R}
 p^{-1}(H^1_{\iota_l})^\vee (\bR)\cong \check{H}^1_{\iota_l} (\bR)\oplus \left(F^0\cap \conj{F^0}\right)
\end{align}
Here $F^0$ refers to the corresponding piece of the filtration of $p^{-1}(H^1_{\iota_l})$ and the construction was explained in \autoref{remark:splittings} (iii).
\autoref{eqn:dir_sum_R} is a direct sum decomposition of $\bR$-vector bundles and $\sigma_\bR$ comes from inverting the $\bR$-isomorphism
$$
W_0^\vee(\bR)\from F^0\cap \conj{F^0}
$$
This means that 
$$\psi_e = (\sigma_{\iota_l} - \sigma_\bR)(e)\in\check{H}^1_{\iota_l}(\bR)$$
is just the $\check{H}^1_{\iota_l}(\bR)$-component of $\phi_e$ in the decomposition \eqref{eqn:dir_sum_R}.

\begin{remark}
 For purposes of comparing metrics, we are using the isomorphism of $\bR$-vector bundles
 $$
 \check{H}^1_{\iota_l}(\bR)\toisom \check{H}^1_{\iota_l}(\bC)/F^0\check{H}^1_{\iota_l}
 $$
 The section we are considering can be viewed as living in either.
 In the second one, it is also holomorphic for the natural holomorphic structure.
\end{remark}

From the property of the dual modified mixed Hodge norm in \eqref{eqn:norm_non_dec}, $\vertiii{\phi_e}$ bounds from above the usual Hodge norm of its $\check{H}^1_{\iota_l}$-component (which is $\psi_e$).
This gives the desired subexponential growth for the Hodge norm $\norm{\psi_e}$.

We must also verify that upon acting by an element $g\in \SL_2\bR$, the function has not increased by more than $C\norm{g}$, for some absolute constant $C$.
This follows from \autoref{eqn:g_t_bdd_mod_norm} and the fact that the norms are $\SO_2(\bR)$-invariant \cite[eq. (41)]{EMM}.

To conclude, the conditions of \cite[Lemma 5.2]{ssimple} are satisfied; therefore the Hodge norm $\norm{\psi_e}$ is constant.
\subsection{Proof of Step 3}

Step 2 showed that $\norm{\psi_e}$ is constant.
To conclude, we now show it must be zero.
Because $\norm{\psi_e}$ is constant, using \cite[Remark 3.3]{ssimple} for the formula for $\delbar\del$ it follows that
$$
0=\delbar \del \norm{\psi_e}^2 = \ip{\Omega \psi_e}{\psi_e} - \ip{\nabla^{Hg}\psi_e}{\nabla^{Hg}\psi}
$$
where $\Omega$ is the curvature of $\check{H}^{-1,0}_{\iota_l}$, which is negative-definite.
We conclude
\begin{align*}
\sigma^\dag \psi_e = 0 \hskip 1cm \textrm{and} \hskip 1cm \nabla^{Hg}\psi_e=0
\end{align*}
Recall that $\sigma:\Omega^{1,0}(\cM)\otimes \check{H}^{0,-1}_{\iota_l}\to\check{H}^{-1,0}_{\iota_l}$ is the second fundamental form of $\check{H}^1_{\iota_l}$ and the above identities hold in all direction on $\cM$, not just the $\SL_2\bR$ direction.
The curvature satisfies $\Omega=\sigma\sigma^\dag$.

Define now a section of $\check{H}^1_{\iota_l}$ by $\alpha := \conj{\psi_e}\oplus \psi_e$.
Then $\alpha$ is flat for the Gauss-Manin connection (since $\nabla^{GM}=\nabla^{Hg}+\sigma+\sigma^\dag$).
But the local system $H^1_{\iota_l}$ is irreducible (see \cite[Thm. 1.5]{Wright_field}) so it has no flat global sections over the affine manifold $\cM$.

Therefore $\alpha=0$, so $\psi_e=0$ and thus $\sigma_{\iota_l}=\sigma_\bR$.
This finishes the proof of \autoref{thm:iota_l_splitting}.
\hfill \qed

\section{Algebraicity and torsion}
\label{sec:algebraicity_torsion}

In this section, we prove the theorems stated in the introduction.
First we combine the results of the previous section to find that a certain twisted version of the Abel-Jacobi map is torsion.
Using this result, we then prove that affine invariant manifolds are algebraic.

\subsection{Combining the splittings}
\label{subsec:combinig_splittings}

\textbf{Setup.}
In the setting of the previous section, over an affine manifold $\cM$ we had an exact sequence of $\bZ$-mixed Hodge structures
$$
0\to W_0 \to E \overset{p}{\to} \bigoplus_\iota H^1_\iota \to 0
$$
We are assuming some fixed $\bZ$-structure on $\oplus_\iota H^1_\iota$.
As was explained in \autoref{subsec:exts_fields} this corresponds to a map
\begin{align}
\label{eqn:def_xi}
\xi:\check{W}_0(\bZ) \to \Jac_\bZ\left(\bigoplus_\iota H^1_\iota\right)
\end{align}
Thus for each element of $\check{W}_0(\bZ)$ we get a section of the bundle of Jacobians.
We are working with a factor of the actual Jacobian, but omit this from the wording.

By \autoref{thm:iota_l_splitting} for each $\iota$ the variation of $\iota(k)$-mixed Hodge structures
$$
0\to W_0/W_0 \cM_{\iota} \to p^{-1}(H^1_{\iota})/W_0 \cM_{\iota} \to H^1_{\iota}\to 0
$$
is split.
In the language of \autoref{subsec:exts_fields} this means that pointwise on $\cM$ the induced map of abelian groups (and $\iota(k)$-vector spaces)
$$
\xi_{\iota}:(W_0 \cM_{\iota})^\perp (\iota(k)) \to \Jac_{\iota(k)}H^1_{\iota}
$$
is in fact the zero map.
Recall $\xi_\iota$ is obtained from $\xi$ by composing with the quotient map
$$
\Jac_\bZ\left(\bigoplus_\iota H^1_\iota\right) \onto \Jac_{\iota(k)} H^1_\iota
$$
and restricting the domain to $(W_0\cM_\iota)^\perp$ (after extending scalars).

Another description of $\xi_\iota$ is as follows.
Given $c^j\in \iota(k)$ and $r_j\in \check{W}_0(\bZ)$ with $\sum_j c^j r_j\in (W_0\cM_\iota)^\perp$ we have
$$
\xi_\iota\left(\sum_j c^j r_j\right) = \left.\sum_j c^j \xi(r_j)\right|_{H^1_\iota}
$$
The next step will be to combine all the above statements, as $\iota$ ranges over all embeddings of $k$ into $\bR$.

\paragraph{{The twisted cycle map.}}
Since $\oplus_\iota H^1_\iota$ has real multiplication by $k$, we have an order $\cO\subseteq k$ which acts by endomorphisms on the bundle of abelian varieties $\Jac_\bZ(\oplus_\iota H^1_\iota)$.
Inside $(W_0\cM_{\iota_0})^\perp$ which is a local system over $\iota_0(k)$, we can choose an $\cO$-lattice, i.e. the $\iota_0(\cO)$-submodule
\begin{align}
\label{eqn:Lambda_iota_def}
\Lambda_{\iota_0}:= \left(\check{W}_0(\bZ)\otimes_{\bZ}\iota_0(\cO) \right)\cap (W_0 \cM_{\iota_0})^\perp
\end{align}
Note that $(W_0\cM_{\iota_0})^\perp\subset \check{W}_0$, since this is the annihilator of $W_0\cM\subset W_0$.
After extending scalars to $\bQ$ we have an isomorphism
$$
\Lambda_{\iota_0}\otimes_\bZ \bQ \toisom (W_0 \cM_{\iota_0})^\perp
$$

\begin{definition}
\label{def:twisted_cycle}
Recall that for $a\in k$ we denote by $\rho(a)$ the corresponding endomorphism of the family of Jacobians and the map $\xi$ was defined in \autoref{eqn:def_xi}.
For $c^j\in \cO$ and $r_j\in \check{W}_0(\bZ)$ we can then define a twisted cycle map
\begin{align*}
 \nu:\Lambda_{\iota_0}&\to \Jac_\bZ\left(\bigoplus_\iota H^1_\iota\right)\\
 \sum_j c^j r_j &\mapsto \sum_j \rho(c^j) \xi(r_j)
\end{align*}
\end{definition}

The following is the main theorem of this section.
It is a generalization of Theorem 3.3 of M\"{o}ller from \cite{Moller_torsion}.

\begin{theorem}
\label{thm:torsion}
 The image of $\nu$ is torsion in the Jacobian.
\end{theorem}

\begin{proof}
 The proof is in two steps.
 First, we show that the image of $\nu$ is zero after we apply the quotient map (where $K$ is a normal closure of $k$)
 $$
 \Jac_\bZ\left(\bigoplus_\iota H^1_\iota\right) \onto \Jac_K\left(\bigoplus_\iota H^1_\iota\right)
 $$
 To finish, we prove that it must have been torsion to begin with.
 
 We have the following spaces and the relation between them which will appear in the proof.
 \begin{eqnarray}
  \begin{tikzcd}[baseline=(current  bounding  box.center), column sep=tiny]
     & \check{W}_0 \arrow{rrr}{\xi} & & & \Jac(\oplus_\iota H^1_\iota) \arrow{dl} \arrow{d} \arrow{dr} & \\
   (W_0 \cM_{\iota_0})^\perp \arrow[hook]{ur} & \cdots \arrow{u} & (W_0\cM_{\iota})^\perp \arrow[hook]{ul}
   &\Jac_{\iota_0(k)}(H^1_{\iota_0}) &\ldots & \Jac_{\iota(k)}(H^1_\iota)
   \label{eqn:diagram}
   \end{tikzcd}
 \end{eqnarray}
 To combine the splittings given by \autoref{thm:iota_l_splitting} we need to extend scalars to the field $K$ which contains all the $\iota(k)$.
 Consider the classifying map obtained from $\xi$
 $$
 \xi_K: \check{W}_0(K)\to \Jac_K\left(\bigoplus_\iota H^1_\iota\right)
 $$
 Given $x=\sum_j c^j r_j \in \Lambda_{\iota_0}$ the splittings from \autoref{thm:iota_l_splitting} say that for any $g\in \Gal(K/\bQ)$ we have
 $$
 \left.\xi_K(gx)\right|_{H^1_{g\iota_0}}=0
 $$
 Recall that to get the splitting we had to apply the Galois action both on $H^1$ and on $W_0$.
 
 Assuming $g\iota_0=\iota_l$ an explicit way to write the above vanishing is
 $$
 \left.\xi_K\left(\sum_j \iota_l(c^j)r_j\right)\right|_{H^1_{\iota_l}}=0
 $$
 Using the $K$-linearity of $\xi_K$ we find
 $$
 \left.\sum_j \iota_l(c^j) \xi_K(r_j)\right|_{H^1_{\iota_l}}=0
 $$
 Taking a direct sum over the different embeddings $\iota_l$ gives that $\nu(x)=0$ in $\Jac_K(\oplus_\iota H^1_\iota)$.
 To see this, recall that the real multiplication action of $k$ on the factor $H^1_\iota$ was by the embedding $\iota$.
 
 To finish, we want to conclude $\nu(x)$ is torsion in the $\bZ$-Jacobian.
 We know that $\nu(x)$ defines a holomorphic section of the corresponding family of abelian varieties.
 We just proved its image lies in the (group-theoretic) kernel of the map
 \begin{align}
 \label{eqn:ker_Jac_map}
 \Jac_\bZ\left(\bigoplus_\iota H^1_\iota\right)\onto \Jac_K \left(\bigoplus_\iota H^1_\iota\right)
 \end{align}
 Viewing the bundle of Jacobians as a flat family of real tori (flat in the sense of the Gauss-Manin connection) we conclude that $\nu(x)$ is itself flat.
 For this, the continuity of $\nu$ would be sufficient.
 
 However, the monodromy acts by parallel transport on a fixed fiber of this family of tori and $\nu(x)$ is invariant by this action.
 Since the underlying local system over $\bQ$ is irreducible (the monodromy acts irreducibly) we conclude that $\nu(x)$ must be among the rational points.
 
 This implies that $\nu(x)$ is a torsion section of the bundle of Jacobians.
\end{proof}

\begin{remark}
\label{rmk:splittings}
 The splitting over $\iota(k)$ of the mixed Hodge structure obtained in \autoref{thm:iota_l_splitting} is equivalent to the composition (see diagram \eqref{eqn:diagram})
 \begin{align}
 \label{eqn:long_composition}
 (W_0\cM_\iota)^\perp \to \check{W_0} \xrightarrow{\xi} \Jac(\oplus_\iota H^1_\iota) \to \Jac_{\iota(k)}(H^1_\iota)
 \end{align}
 being the zero map.
 The torsion condition obtained in \autoref{thm:torsion} implies that the twisted cycle map $\nu:\Lambda_{\iota_0}\to \Jac_\bZ(\oplus_\iota H^1_\iota)$ is torsion.
 In particular, $A\cdot \nu$ is the zero map for some integer $A\neq 0$.
 Note that the twisted cycle map $\nu$ involves the maps in \eqref{eqn:long_composition} ranging over all embeddings $\iota$.
 
 \noindent
 The condition that $A\cdot \nu$ is the zero map implies, in particular, that the composition defined in \eqref{eqn:long_composition} is also the zero map, as $\iota$ ranges over all embeddings.
 Indeed, these maps are the components of $\nu$ after an extension of scalars to the appropriate number field.

 \noindent Thus, the torsion condition in \autoref{thm:torsion} implies the splitting property in \autoref{thm:iota_l_splitting}.
\end{remark}

\subsection{Algebraicity in the general case}

\paragraph{Setup.} We keep the notation from the previous section.
$\cM$ denotes an affine invariant manifold, $T\cM_{\iota_0}$ its tangent bundle and $W_0\cM_{\iota_0}$ its intersection with the purely relative subbundle.
We have the decomposition of the Hodge bundle
$$
H^1=\left(\bigoplus_\iota H^1_\iota\right)\oplus V
$$
Recall also that $\cM$ lives in a stratum $\cH$ and we have the tautological section
$$
\omega:\cH\to T\cH
$$
This section is algebraic and even defined over $\bQ$ (see \autoref{rmk:field_defn}).
Combined with the Gauss-Manin connection it gives the local flat coordinates on the stratum $\cH$.

\begin{theorem}
\label{thm:algebraic}
The affine invariant manifold $\cM$ is a quasi-projective algebraic subvariety of $\cH$, defined over $\conj{\bQ}$.
\end{theorem}

\begin{proof}
The proof is similar to that of \autoref{prop:alg_special_case}.
In particular, as per \autoref{rmk:self_inters} it is carried out in the open subset of $\cM$ where there are no self-intersections.
We first define an algebraic variety $\cN'''$ which has the same properties as $\cM$ and contains it. 
Then we check that one of its irreducible components is contained in $\cM$ and thus has to coincide with it.

First, let $\cN'\subseteq \cH$ be the locus where $(X,\omega)$ admits real multiplication of the same type as on $\cM$, with $\omega$ as an eigenform.
On $\cN'$ we also have the map $\nu$ described in \autoref{def:twisted_cycle}
$$
\nu:\Lambda_{\iota_0} \to \Jac_\bZ\left(\bigoplus_\iota H^1_{\iota}\right)
$$
By \autoref{thm:torsion}, over $\cM$ the map lands in the torsion so there exists $A\in \bN$ such that
$A\cdot\nu\equiv 0$ everywhere on $\cM$.
Let $\cN''\subseteq \cN'$ be the sublocus where the equation $A\nu=0$ holds on $\cN'$.
This is again algebraic, defined over $\conj{\bQ}$, and contains $\cM$.

Recall now that we are working in a finite cover where the local system $W_0$ is trivial and so $W_0\cM_{\iota_0}$ is globally defined.
Over $\cN''$ the condition $A\cdot \nu=0$ implies (see \autoref{rmk:splittings}) that the exact sequence of mixed Hodge structures splits (over $\iota_0(k)$)
$$
0\to W_0/W_0\cM_{\iota_0} \to p^{-1}(H^1_{\iota_0})/W_0\cM_{\iota_0}\to H^1_{\iota_0} \to 0
$$
The splitting subbundle in $p^{-1}(H^1_{\iota_0})/W_0 \cM_{\iota_0}$ is unique and itself algebraic.
Denote this bundle by $T'$.
It also gives a local system and it is the candidate for the tangent space.

Recall we had the tautological section $\omega:\cH\to T\cH$ and let $\cN'''\subseteq\cN''$ be the locus where $\omega\in T'$.
This is an algebraic variety over $\conj{\bQ}$ by construction, and we claim $\cM$ coincides with the irreducible component of $\cN'''$ containing it.

To see this, first note that $T'=T\cM$ on $\cM$.
However, locally near a point of $\cM$ the condition $\omega\in T'$ is the same as being on $\cM$.
Indeed, requiring $\omega$ to lie in some local subsystem is the same as requiring the flat surface to be in a linear subspace in local period coordinates (see \autoref{rmk:loc_syst_per_coord}).

Finally, note that $T'$ is a subquotient of $T\cH$, but the condition $\omega\in T'$ still makes sense.
It is understood as $\omega$ belonging to the preimage of $T'$ in $T\cH$.
This completes the proof of algebraicity.
\end{proof}

\begin{remark}
\label{rmk:affine_mnfds_dimension}
 Let us explain how affine manifolds are distinguished among loci of real multiplication, since usually these are not $\GL_2^+\bR$-invariant.
 
 Suppose given a subvariety $\cN'\subset \cH$ parametrizing $(X,\omega)$ admitting real multiplication by $k$ with $\omega$ an eigenform.
 Given a $k$-local system $W_0T\subset W_0$ (thinking of it as $W_0\cM_{\iota_0}$), further require that the quotient mixed Hodge structure splits as in the above theorem.
 Let the splitting be given by a bundle (and also local system) denoted $T$.
 Finally, require that $\omega$ lies in $T$.
 
 This provides us with a locus $\cN\subseteq \cH$.
 Note that the requirement $\omega \in T$ implies that the Zariski tangent bundle of $\cN$, denoted $T\cN$, is contained in $T$.
 The variety $\cN$ will also be $\SL_2\bR$-invariant if and only if we have the equality $T\cN=T$.
\end{remark}

\begin{remark}[\textbf{On fields of algebraic definition}]
 \label{rmk:field_defn}\leavevmode
 \begin{enumerate}
 \item[(i)] For the purposes of this discussion, a variety is ``defined over a field $K$'' if it can be described in a projective space as the zero locus of polynomials with coefficients in $K$.
 It is quasi-projective if it can be described in a projective space as the locus where a given collection of polynomials vanish, and another collection doesn't vanish.
 A map between varieties (in particular a section of a bundle) is ``defined over $K$'' if it can be described using polynomials with coefficients in $K$.
 
 \item[(ii)] A stratum $\cH$ is a (quasi-projective) algebraic variety defined over $\bQ$.
 To see this, recall that the moduli space of genus $g$ Riemann surfaces $\cM_g$ is defined over $\bQ$, and so is the Hodge bundle $\cH_g\xrightarrow{\pi}\cM_g$.
 Moreover, the pullback of the Hodge bundle $\pi^*\cH_g$ to $\cH_g$ has the tautological section $\omega:\cH_g\to \pi^*\cH_g$, also defined over $\bQ$.
 Next, letting $\cC_g\to \cM_g$ be the universal bundle of Riemann surfaces, the Hodge bundle admits a map $Div:\cH_g\to \Sym^{2g-2}_{\cM_g}\cC_g$ which takes a $1$-form to its zero locus.
 The space $\Sym^{2g-2}_{\cM_g}\cC_g$ parametrizes $2g-2$-tuples of (not necessarily distinct) points on the same fiber of the universal family.
 This space is defined over $\bQ$ and has a stratification, also defined over $\bQ$, depending on multiplicities.
 A stratum of $\cH_g$ is then the preimage of one of the strata on $\Sym^{2g-2}_{\cM_g}\cC_g$.
 
 The variety thus obtained need not be connected.
 To distinguish connected components, one might apriori need to extend the base field $\bQ$.
 However, from the classification of connected components due to Kontsevich \& Zorich \cite{KZ_components}, each connected component can be described by an algebraic condition also defined over $\bQ$.
 Indeed, the spin invariant of a square root of the canonical bundle is invariant by Galois conjugation, and so is the property of being hyperelliptic.
 
 \item[(iii)] The tautological section $\omega$ of the Hodge bundle restricted to a stratum is then also defined over $\bQ$.
 Moreover, the stratum carries a universal family of Riemann surfaces, and also the canonical set of marked points corresponding to the zeros of $\omega$.
 The vector bundle $H^1_{rel}$ has a description as the algebraic de Rham cohomology of this family of Riemann surfaces, and is thus itself defined over $\bQ$.
 The natural map from the Hodge bundle $\pi^*\cH_g$ to $H^1_{rel}$ is also defined over $\bQ$, and so $\omega:\cH\to H^1_{rel}\cong T\cH$ is defined over $\bQ$.
 See also \cite[Section 2]{Moller_Linear_mnfds} for a detailed discussion of these constructions.
 
 \item[(iv)] The proof of \autoref{thm:algebraic} involved the locus $\cN'$ where a factor of the Jacobian had real multiplication by $\cO$.
 In $\cA_g$, this locus is defined over $\bQ$, but to select the components which have the $\cO$-module structure on $H^1$ as on $\cM$, we need to pass to some finite extension of $\bQ$.
 Hence, the locus $\cN'$ is defined over $\overline{\bQ}$.
 Next, to define the locus $\cN''$ by imposing the torsion condition $A\cdot \nu \equiv 0$ required a finite cover where the zeros are labeled.
 This requires another finite extension of the base field.
 
 \item[(v)] An intersection of two varieties (e.g. imposing the condition that $\omega$ lies in a subbundle) is defined over a field which contains the defining fields for both varieties.
 \end{enumerate}
\end{remark}

 The above discussion explains why affine invariant manifolds are quasi-projective varieties defined over $\conj{\bQ}$.
 Acting by the Galois group of $\bQ$ on the defining equations inside $\cH$ produces new quasi-projective varieties.
 These will also be affine invariant manifolds.
 
\begin{corollary}
 The group $\Gal(\bQ)$ acts on the set of affine invariant manifolds.
\end{corollary}
\begin{proof}
 From the proof of \autoref{thm:algebraic} (see also the discussion in \autoref{rmk:affine_mnfds_dimension}) an affine manifold $\cM$ is defined by the following set of algebraic conditions.
 \begin{itemize}
  \item The parametrized Riemann surfaces have a factor with real multiplication by $\cO$, and $\omega$ as an eigenform.
  \item There is a sublocal system $\Lambda_{\iota_0}\subset \check{W}_0$ such that its image under the twisted Abel-Jacobi map is torsion of some fixed degree, i.e. $A\cdot \nu\equiv 0$ for some integer $A\neq 0$.
  \item There is an upper bound for the dimension of this locus inside the stratum $\cH$ (computed from the ranks of the objects above) and the dimension of $\cM$ achieves this bound.
 \end{itemize}
 Each of the above conditions persists when the Galois group of $\bQ$ acts on the defining equations.
 Note that $\omega$ is defined with $\bQ$-coefficients, and so if it is an eigenform on one locus, it will also be an eigenform in the Galois-conjugate locus. 
\end{proof}

\begin{remark}
 The results of Wright \cite{Wright_field} show that in local period coordinates on $H^1_{rel}$ an affine manifold is described by linear equations with coefficients in the number field $k$.
 These equations and field of definition are usually not related to the algebraic equations and respective fields.
 In particular, acting by the Galois group on the linear equations would typically not produce another affine manifold.
 
 For comparison, the field of affine definition of square-tiled surfaces is $\bQ$.
 However, algebraically these can be quite rich -- M\"{o}ller \cite[Thm. 5.4]{Moller_GT} showed the action of the Galois group is faithful on the corresponding \Teichmuller curves.
\end{remark}



\bibliographystyle{sfilip}
\bibliography{bibliography_splitting_mhs_affine_mnfds}

\end{document}